\documentclass[10pt,twoside]{article}
\usepackage{mathrsfs}
\usepackage{amssymb}
\usepackage{amsmath,color}
\usepackage[all]{xy}
\usepackage{amsthm}
\usepackage{url}
\usepackage{indentfirst}
\usepackage{cite}
\usepackage[numbers,sort&compress]{natbib}
\numberwithin{equation}{section}

\newtheorem{thm}{Theorem}[section]
\newtheorem{cor}[thm]{Corollary}
\newtheorem{lem}[thm]{Lemma}

\newtheorem{prop}[thm]{Proposition}

\theoremstyle{definition}
\newtheorem{defn}[thm]{Definition}
\newtheorem{rem}[thm]{Remark}
\newtheorem{example}[thm]{Example}

\newcommand{\pf}{\noindent\begin {proof}}
\newcommand{\epf}{\end{proof}}

\newcommand{\Ext}{\mbox{\rm Ext}}
\newcommand{\Hom}{\mbox{\rm Hom}}

\newcommand{\Coker}{\mbox{\rm Coker}}
\def\Im{\mathop{\rm Im}\nolimits}
\def\Ker{\mathop{\rm Ker}\nolimits}
\def\Coker{\mathop{\rm Coker}\nolimits}

\def\mod{\mathop{\rm mod}\nolimits}
\def\Mod{\mathop{\rm Mod}\nolimits}

\def\id{\mathop{\rm id}\nolimits}
\def\pd{\mathop{\rm pd}\nolimits}

\def\sup{\mathop{\rm sup}\nolimits}

\def\Hom{\mathop{\rm Hom}\nolimits}
\def\Ext{\mathop{\rm Ext}\nolimits}
\def\sup{\mathop{\rm sup}\nolimits}
\def\lim{\mathop{\underrightarrow{\rm lim}}\nolimits}

\title{ \bf Gorenstein Projective Objects in Comma Categories
}
\vspace{0.2cm}

\author{Yeyang Peng, Rongmin Zhu, Zhaoyong Huang\thanks{E-mail address:  pengyy@smail.nju.edu.cn, rongminzhu@hotmail.com, huangzy@nju.edu.cn.
}  \\
{\it \footnotesize  Department of Mathematics, Nanjing University, Nanjing 210093, Jiangsu Province, P. R. China}}
\date{ }
\begin{document}

\maketitle

\begin{abstract}
Let $\mathcal{A}$ and $\mathcal{B}$ be abelian categories and $\mathbf{F}:\mathcal{A}\to \mathcal{B}$
an additive and right exact functor which is perfect, and let $(\mathbf{F},\mathcal{B})$ be
the left comma category. We give an equivalent characterization of Gorenstein projective objects in
$(\mathbf{F},\mathcal{B})$ in terms of Gorenstein projective objects in $\mathcal{B}$ and $\mathcal{A}$.
We prove that there exists a left recollement of the stable category of the subcategory of $(\mathbf{F},\mathcal{B})$
consisting of Gorenstein projective objects modulo projectives relative to the same kind of stable categories in $\mathcal{B}$ and
$\mathcal{A}$. Moreover, this left recollement can be filled into a recollement when $\mathcal{B}$
is Gorenstein and $\mathbf{F}$ preserves projectives.
\end{abstract}

\pagestyle{myheadings}
\markboth{\rightline {\scriptsize Yeyang Peng, Rongmin Zhu and Zhaoyong Huang}}
         {\leftline{\scriptsize Gorenstein projective objects in comma category}}

\baselineskip=18pt
\section{Introduction}

As a generalization of finitely generated projective modules, Auslander and Bridger \cite{AB} introduced
finitely generated modules of Gorenstein dimension zero over a commutative noetherian local ring. Then
Enochs and Jenda \cite{EJ1} generalized it to Gorenstein
projective modules (not necessarily finitely generated)
over an arbitrary ring. The properties of Gorenstein projective modules and
related modules have been studied widely, see \cite{AS,AB,ECT,EJ1,EJ2,PSS,XZ,Z}
and references therein.

Let $\Lambda$ and $\Gamma$ be arbitrary rings and $M$ a (finitely generated) $(\Lambda,\Gamma)$-bimodule,
and let $T:=\left(\begin{array}{cc}\Lambda & M \\
0 & \Gamma\end{array}\right)$ be the upper triangular ring. Recall from \cite{Z}
that the $(\Lambda,\Gamma)$-bimodule $M$ is called {\it compatible} if the following
two conditions are satisfied: (1) if $Q^{\bullet}$ is an exact sequence of finitely
generated projective $\Gamma$-modules, then $M\otimes_{\Gamma}Q^{\bullet}$ is exact;
and (2) if $P^{\bullet}$ is a complete finitely generated $\Lambda$-projective resolution,
then $\Hom_{\Lambda}(P^{\bullet},M)$ is exact.
Let $\Lambda$ and $\Gamma$ be artin algebras and the bimodule
${_{\Lambda}M_{\Gamma}}$ compatible. Then finitely generated Gorenstein projective $T$-modules
can be constructed from finitely generated Gorenstein projective $\Lambda$-modules and finitely generated
Gorenstein projective $\Gamma$-modules (\cite[Theorem 1.4]{Z}).
Moreover, there exists a left recollement of the stable category
$\underline{\mathcal{GP}(T)}$ of the category of finitely generated Gorenstein projective
$T$-modules modulo projectives relative to $\underline{\mathcal{GP}(\Lambda)}$
and $\underline{\mathcal{GP}(\Gamma)}$ (\cite[Theorem 3.3]{Z}),
and this left recollement can be filled into a recollement when $T$
is Gorenstein and $_{\Lambda}M$ is projective (\cite[Theorem 3.5]{Z}).
Under some conditions, Enochs, Cort\'es-Izurdiaga and Torrecillas proved that
$T$ is (strongly) CM-free if and only if so are $\Lambda$ and $\Gamma$
(\cite[Theorem 4.1]{ECT}).

Let $\mathcal{A}$ and $\mathcal{B}$ be abelian categories and $\mathbf{F}:\mathcal{A}\to \mathcal{B}$
an additive functor. The left comma category $(\mathbf{F},\mathcal{B})$ was introduced in \cite{FGR}.
Note that module categories of upper triangular matrix rings are comma categories and that
the left comma category $(\mathbf{F},\mathcal{B})$ is abelian if $\mathbf{F}$ is right exact
(\cite{FGR, Ps}). The aim of this paper is to generalize the results
mentioned above from module categories of upper triangular matrix rings to comma categories.
The paper is organized as follows.

In Section 2, we give some terminology and some preliminary results.

For an abelian category $\mathcal{A}$,
we use $\mathcal{GP(A)}$ to denote the subcategory of $\mathcal{A}$ consisting of Gorenstein projective objects,
and use $\underline{\mathcal{\mathcal{GP(A)}}}$ to denote the stable category of $\mathcal{GP(A)}$ modulo projectives.
Motivated by the definition of compatible bimodules \cite{Z}, we introduce the so-called
{\it perfect} functors between abelian categories (Definition \ref{def3.3}).
Let $\mathcal{A}$ and $\mathcal{B}$ be abelian categories and $\mathbf{F}:\mathcal{A}\to \mathcal{B}$
an additive and right exact functor such that $\mathbf{F}$ is perfect,
and let $(\mathbf{F},\mathcal{B})$ be the left comma category.
Then we give an equivalent characterization of Gorenstein projective objects in $(\mathbf{F},\mathcal{B})$
in terms of Gorenstein projective objects in $\mathcal{B}$ and $\mathcal{A}$.

\begin{thm}\label{th1.1} {\rm (Theorem \ref{th3.5})}
The following statements are equivalent for an object ${{Y} \choose{X}}_{\phi}$ in $(\mathbf{F},\mathcal{B})$.
\begin{itemize}
\item[(1)] ${{Y} \choose{X}}_{\phi}\in\mathcal{GP}((\mathbf{F},\mathcal{B}))$.
\item[(2)] $\phi:\mathbf{F}{Y}\rightarrow {X}$ is injective in $\mathcal{B}$,
$\Coker\phi\in\mathcal{GP}(\mathcal{B})$ and $Y\in\mathcal{GP}(\mathcal{A})$.
\end{itemize}
\end{thm}

As an application, we get that the Gorenstein projective objects coincide with projective objects in
$(\mathbf{F},\mathcal{B})$ if and only if both $\mathcal{A}$ and $\mathcal{B}$ also possess the same property
(Corollary \ref{cor3.9}).

In Section 4, we prove the following
\begin{thm}\label{th1.2} {\rm (Theorem \ref{th4.4})}
There exists a left recollement
$$\xymatrix{\underline{\mathcal{GP}(\mathcal{B})}\ar@<-1ex>[rr]!R|(.4){i_{*}}
&&\ar@<-1ex>[ll]!R|(.6){i^{*}}\ar@<-1ex>[rr]!R|(.4){j^{*}}\underline{\mathcal{GP}((\mathbf{F},\mathcal{B}))}
&&\ar@<-1ex>[ll]!R|(.6){j_{!}}\underline{\mathcal{GP}(\mathcal{A})}.}$$
\end{thm}
Moreover, this left recollement can be filled into a recollement when $\mathcal{B}$
is Gorenstein and $\mathbf{F}$ preserves projectives (Theorem \ref{thm4.7}).

\section{Preliminaries}\label{pre}

In this section, we give some notions and some preliminary results.

Let $\mathcal{A}$ be an abelian category and all subcategories of $\mathcal{A}$
are full and closed under isomorphisms. We use $\mathcal{P(A)}$ and $\mathcal{I(A)}$ to denote
the subcategories of $\mathcal{A}$ consisting of projective and injective objects respectively.
For an object $A$ in $\mathcal{A}$, $\pd_{\mathcal{A}}A$ and $\id_{\mathcal{A}}A$ are the projective
and injective dimensions of $A$ respectively.
For a subcategory $\mathcal{X}$ of $\mathcal{A}$, set
$$\pd_{\mathcal{A}}\mathcal{X}:=\sup\{\pd_{\mathcal{A}}A\mid A\in\mathcal{X}\}\
{\rm and}\ \id_{\mathcal{A}}\mathcal{X}:=\sup\{\id_{\mathcal{A}}A\mid A\in\mathcal{X}\}.$$

By using a standard argument, we have the following generalized horseshoe lemma.

\begin{lem}\label{lem2.1}
Let $\mathcal{A}$ be an abelian category and
$$0\longrightarrow {Y}\stackrel{f}\longrightarrow {X}\stackrel{g}\longrightarrow{Z}\longrightarrow 0$$
an exact sequence in $\mathcal{A}$.
\begin{itemize}
\item[(1)] Let $$\xymatrix{Y\ar[r]^{c^{-1}}&C^{0}\ar[r]^{c^{0}}&C^{1}\ar[r]^{c^{1}}&\cdots}$$
be a complex and
$$\xymatrix{0\ar[r]&Z\ar[r]^{d^{-1}}&D^{0}\ar[r]^{d^{0}}&D^{1}\ar[r]^{d^{1}}&\cdots}$$ an exact sequence in $\mathcal{A}$.
If $\Ext_{\mathcal{A}}^{1}(\Ker d^{i},C^{i})=0$ for any $i\geq 0$, then there exist morphisms
$$\partial^{-1}={d^{-1}g \choose \sigma^{-1}}:\xymatrix{{X}\ar[r]&D^{0}\oplus C^{0}}\ and\ \partial^{i}=\left(\begin{array}{cc}
 d^{i}&0\\
 \sigma^{i}&c^{i}
\end{array}\right):\xymatrix{D^{i}\oplus C^{i}\ar[r]&D^{i+1}\oplus C^{i+1}}$$ with $\sigma^{i}:D^{i}\rightarrow C^{i+1}$
for any $i\geq 0$, such that
$$\xymatrix{0\ar[r]&X\ar[r]^{\partial^{-1}}&D^{0}\oplus C^{0}\ar[r]^{\partial^{0}}&D^{1}\oplus C^{1}\ar[r]^{\partial^{1}}&\cdots
\ar[r]^{\partial^{i-1}}&D^{i}\oplus C^{i}\ar[r]^{\partial^{i}}&\cdots}$$
is a complex in $\mathcal{A}$ and the following diagram with exact rows
\begin{center}$\xymatrix{&0\ar[d]&0\ar[d]&0\ar[d]&\\
                         0\ar[r]&Y\ar[r]^{f}\ar[d]^{c^{-1}}&X\ar[r]^{g}\ar[d]^{\partial^{-1}}&Z\ar[r]\ar[d]^{d^{-1}}&0\\
                         0\ar[r]&C^{0}\ar[r]\ar[d]^{c^{0}}&D^{0}\oplus C^{0}\ar[r]\ar[d]^{\partial^{0}}&D^{0}\ar[r]\ar[d]^{d^{0}}&0\\
                         0\ar[r]&C^{1}\ar[r]\ar[d]^{c^{1}}&D^{1}\oplus C^{1}\ar[r]\ar[d]^{\partial^{1}}&D^{1}\ar[r]\ar[d]^{d^{1}}&0\\
                         &\vdots & \vdots & \vdots &}$\end{center}
commutes. Moreover, the middle column is exact if and only if the left column is exact.
\item[(2)] Let $$\xymatrix{\cdots \ar[r]^{e_{2}}& E_{1}\ar[r]^{e_{1}}& E_{0}\ar[r]^{e_{0}}&Y\ar[r]&0}$$ be an exact sequence
and $$\xymatrix{\cdots \ar[r]^{f_{2}}& F_{1}\ar[r]^{f_{1}}&}\xymatrix{F_{0}\ar[r]^{f_{0}}&Z\ar[r]&0}$$ a complex in $\mathcal{A}$.
If $\Ext_{\mathcal{A}}^{1}(F_{i},\Im e_{i})=0$ for any $i\geq 0$, then there exist morphisms
$$\partial_{0}=(\pi_{0}, fe_{0}):\xymatrix{F^{0}\oplus E^{0}\ar[r]& X}\ and\ \partial_{i}=\left(\begin{array}{cc}
 f_{i}& 0 \\
 \pi_{i} & e_{i}
\end{array}\right):\xymatrix{F_{i}\oplus E_{i}\ar[r]&F_{i-1}\oplus E_{i-1}}$$ with $\pi_{i}:F_{i}\rightarrow E_{i-1}$
 for any $i\geq 1$, such that
 \begin{center}$\xymatrix{\cdots \ar[r]^{\partial_{i+1}}&F_{i}\oplus E_{i}\ar[r]&\cdots\ar[r]^{\partial_{2}}
 &F_{1}\oplus E_{1}\ar[r]^{\partial_{1}}&F_{0}\oplus E_{0}\ar[r]^{\partial_{0}}&X\ar[r]&0}$\end{center}
 is a complex in $\mathcal{A}$ and the following diagram with exact rows
 \begin{center}$\xymatrix{&\vdots\ar[d] & \vdots\ar[d] & \vdots\ar[d] &\\
 0\ar[r]&E_{1}\ar[r]\ar[d]^{e_{1}}&F_{1}\oplus E_{1}\ar[r]\ar[d]^{\partial_{1}}&F_{1}\ar[r]\ar[d]^{f_{1}}&0\\
 0\ar[r]&E_{0}\ar[r]\ar[d]^{e_{0}}&F_{0}\oplus E_{0}\ar[r]\ar[d]^{\partial_{0}}&F_{0}\ar[r]\ar[d]^{f_{0}}&0\\
 0\ar[r]&{Y}\ar[r]^{f}\ar[d]&{X}\ar[d]\ar[r]^{g}&{Z}\ar[d]\ar[r]&0\\
 &0 & 0 & 0 &}$\end{center}
commutes. Moreover, the middle column is exact if and only if the right column is exact.
\end{itemize}
\end{lem}

\begin{defn}(\cite{FGR})\label{def2.2}
Let $\mathcal{A}$ be an abelian category and $\mathbf{F}:\mathcal{A}
\longrightarrow \mathcal{A}$ an additive endofunctor. The {\it right trivial extension}
of $\mathcal{A}$ by $\mathbf{F}$, denoted by $\mathcal{A} \ltimes
\mathbf{F}$, is defined as follows.
An {\it object} in $\mathcal{A} \ltimes \mathbf{F}$ is a morphism
$\alpha:\mathbf{F}A \longrightarrow A$ for an object
$A$ in $\mathcal{A}$ such that $\alpha\cdot\mathbf{F}(\alpha)=0$; and a
{\it morphism} in $\mathcal{A} \ltimes \mathbf{F}$ is a pair $(\mathbf{F}\gamma,\gamma)$
of morphisms in $\mathcal{A}$ such that the following diagram
$$\xymatrix{\mathbf{F}A\ar[r]^{\mathbf{F}\gamma}\ar[d]^{\alpha}&\mathbf{F}A'\ar[d]^{\alpha'}\\
A\ar[r]^{\gamma}&A'}$$
is commutative.
\end{defn}

\begin{defn}(\cite{FGR})\label{def2.3}
Let $\mathcal{A}$ and $\mathcal{B}$ be abelian categories and $\mathbf{F}:\mathcal{A} \longrightarrow \mathcal{B}$
an additive functor. We define the {\it left comma category} $(\mathbf{F},\mathcal{B})$ as follows.
The {\it objects} of the category are ${{A} \choose {B}}_{\phi}$ with $A\in\mathcal{A}$, $\in\mathcal{B}$ and
$\phi\in\Hom_{\mathcal{B}}(\mathbf{F}{A},{B})$; and the {\it morphisms} of the category are pairs
$\alpha\choose\beta$ of morphisms in $\mathcal{A}\times \mathcal{B}$ such that the following diagram
$$\xymatrix{\mathbf{F}{A}\ar[r]^{\mathbf{F}\alpha}\ar[d]^{\phi}&\mathbf{F}{A'}\ar[d]^{\phi'}\\
{B}\ar[r]^{\beta}&{B'}}$$
is commutative.
\end{defn}

\begin{rem}(\cite[Section 1]{FGR})\label{rem2.4}
\begin{itemize}
\item[(1)]
The above functor $\mathbf{F}$ induces a functor $\mathbf{\widetilde{F}}:\mathcal{A} \times \mathcal{B}
\rightarrow \mathcal{A} \times \mathcal{B}$ by $\mathbf{\widetilde{F}}({A},{B})=(0,\mathbf{F}{A})$
and $\mathbf{\widetilde{F}}(\alpha,\beta)=(0,\mathbf{F}\alpha)$. It is not difficult to show that
$(\mathbf{F},\mathcal{B})$ and $(\mathcal{A} \times \mathcal{B})\ltimes\mathbf{\widetilde{F}}$ are isomorphic.
\item[(2)]
If $\mathbf{F}$ is right exact, then $\mathcal{A}\ltimes \mathbf{F}$ is abelian.
Hence by the isomorphism of
$(\mathcal{A} \times \mathcal{B})\ltimes\mathbf{\widetilde{F}}$ and $(\mathbf{F},\mathcal{B})$,
it is clear that if $\mathbf{F}$ is right exact, then $(\mathbf{F},\mathcal{B})$ is abelian.
\end{itemize}
\end{rem}

Recall that a sequence in $\mathcal{A}$ is called {\it $\Hom(-,\mathcal{P(A)})$-exact}
if it is exact after applying the functor $\Hom(-,{P})$ for any $P\in\mathcal{P(A)}$.

\begin{defn}\label{def2.5} {\rm(\cite{EJ2})}
An object $G\in\mathcal{A}$ is called {\it Gorenstein projective} if
there exists a $\Hom(-,\mathcal{P(A)})$-exact exact sequence
$$\cdots\rightarrow{Q_{1}}\rightarrow{Q_{0}}\rightarrow{Q^{0}}\rightarrow{Q^{1}}\to\cdots$$ in $\mathcal{A}$
will all $Q_i,Q^i$ projective,
such that ${G} \cong \Im ({Q_{0}}\rightarrow{Q^{0}})$; in this case, this exact sequence is called a
{\it complete $\mathcal{A}$-projective resolution} of ${G}$.
\end{defn}

We write $\mathcal{GP}(\mathcal{A}):=\{G\in\mathcal{A}\mid G$ is Gorenstein projective$\}$.
It is well known that $\mathcal{GP}(\mathcal{A})$ is a Frobenius category such that each object in $\mathcal{P(A)}$
is projective-injective in $\mathcal{GP}(\mathcal{A})$ and its stable category $\underline{\mathcal{GP}(\mathcal{A})}$
modulo $\mathcal{P(A)}$ is a triangulated category.

\section{Gorenstein projective objects}\label{Goren}

From now on, assume that $\mathcal{A}$ and $\mathcal{B}$ are abelian categories and
$\mathbf{F}:\mathcal{A} \longrightarrow \mathcal{B}$ is an additive and right exact functor, and
$(\mathbf{F},\mathcal{B})$ is the left comma category.
Then $(\mathbf{F},\mathcal{B})$ is abelian by Remark \ref{rem2.4}(2).
It is known from \cite{FGR} that the projective object in $(\mathcal{A} \times \mathcal{B}) \ltimes
\mathbf{\widetilde{F}}$ is of the form
$(\mathbf{\widetilde{F}}({P,Q})\oplus\mathbf{\widetilde{F}^{2}}({P,Q})\longrightarrow ({P,Q})\oplus\mathbf{\widetilde{F}}({P,Q}))$
with ${P}$ projective in $\mathcal{A}$ and ${Q}$ projective in $\mathcal{B}$.
So by Remark \ref{rem2.4}(1), we have the following

\begin{lem}\label{lem3.1}
The projective object in $(\mathbf{F},\mathcal{B})$ is of the form
${0 \choose {Q}}\oplus{{P} \choose \mathbf{F}{P}}$ with ${P}$ projective in $\mathcal{A}$ and ${Q}$ projective
in $\mathcal{B}$.
\end{lem}

The following result generalizes \cite[Proposition 2.8(1)]{ECT}.

\begin{prop}\label{prop3.2}
Let ${M_{1} \choose M_{2}}$ be an object in $(\mathbf{F},\mathcal{B})$.
If $\pd_{\mathcal{B}}\mathbf{F}\mathcal{P(A)}<\infty$, then
$\pd_{(\mathbf{F},\mathcal{B})}{M_{1} \choose M_{2}}<\infty$ if and only if $\pd_{\mathcal{A}}M_{1}<\infty$
and $\pd_{\mathcal{B}}M_{2}<\infty$.
\end{prop}

\begin{proof}
Let $\pd_{(\mathbf{F},\mathcal{B})}{M_{1} \choose M_{2}}<\infty$. Then by Lemma \ref{lem3.1},
we have the following exact sequence of finite length
$$\xymatrix@C=0.5cm{
  0 \ar[r] & {0 \choose P_{n}}\oplus{Q_{n} \choose \mathbf{F}Q_{n}} \ar[r] & \cdots \ar[r]
  & {0 \choose P_{2}}\oplus{Q_{2} \choose \mathbf{F}Q_{2}} \ar[r]
  & {0 \choose P_{1}}\oplus{Q_{1} \choose \mathbf{F}Q_{1}} \ar[r]
  & {0 \choose P_{0}}\oplus{Q_{0} \choose \mathbf{F}Q_{0}} \ar[r] & {M_{1} \choose M_{2}} \ar[r] & 0 }$$
in $(\mathbf{F},\mathcal{B})$ with all $Q_i$ projective in $\mathcal{A}$ and all $P_i$ projective in $\mathcal{B}$.
Hence we have exact sequences
$$\xymatrix@C=1cm{
  0 \ar[r] &  Q_{n}\ar[r] & \cdots \ar[r]& Q_{2} \ar[r] & Q_{1} \ar[r] & Q_{0} \ar[r] & M_{1} \ar[r] & 0, }\eqno{(3.1)}$$
$$\xymatrix@C=0.5cm{
  0 \ar[r] &  P_{n}\oplus\mathbf{F}Q_{n}\ar[r] & \cdots \ar[r]& P_{2}\oplus\mathbf{F}Q_{2} \ar[r]
  & P_{1}\oplus\mathbf{F}Q_{1} \ar[r] & P_{0}\oplus\mathbf{F}Q_{0} \ar[r] & M_{2} \ar[r] & 0 }\eqno{(3.2)}$$
in $\mathcal{A}$ and $\mathcal{B}$ respectively. By (3.1), we have $\pd_{\mathcal{A}}M_{1}<\infty$.
Since $\pd_{\mathcal{B}}\mathbf{F}Q_i<\infty$ for any $0\leq i\leq n$ by assumption,
we have $\pd_{\mathcal{B}}M_{2}<\infty$ by (3.2).

Conversely, assume $\pd_{\mathcal{A}}M_{1}<\infty$ and $\pd_{\mathcal{B}}M_{2}<\infty$. Let
$$\xymatrix@C=1cm{
  0 \ar[r] &  Q_{n}\ar[r]^{\delta^{1}_{n}} & \cdots \ar[r]^{\delta^{1}_{3}} & Q_{2} \ar[r]^{\delta^{1}_{2}}
  & Q_{1} \ar[r]^{\delta^{1}_{1}} & Q_{0} \ar[r]^{\delta^{1}_{0}}  & M_{1} \ar[r] & 0 }$$
be a projective resolution of $M_{1}$ in $\mathcal{A}$. Then $\pd_{\mathcal{A}}K^{1}_{i}<\infty$,
where $K^{1}_{i}:=\Ker \delta^{1}_{i-1}$ for any $1\leq i\leq n+1$.
Fix a projective presentation $P_{0}\twoheadrightarrow M_{2}$
of $M_{2}$ in $\mathcal{B}$. Then we can construct a projective presentation ${Q_{0} \choose
P_{0}\oplus\mathbf{F}Q_{0} } \twoheadrightarrow {M_{1} \choose M_{2}}$ of ${M_{1} \choose M_{2}}$ in $(\mathbf{F},\mathcal{B})$.
If ${K^{1}_{1}\choose K^{2}_{1}}$ is its kernel, then there exists an exact sequence
$$\xymatrix@C=0.5cm{0 \ar[r] & K^{2}_{1} \ar[r] & P_{0}\oplus\mathbf{F}Q_{0} \ar[r] & M_{2} \ar[r] & 0 }$$
in $\mathcal{B}$. Because $\pd_{\mathcal{B}}\mathbf{F}Q_{0}<\infty$ by assumption, we have $\pd_{\mathcal{B}}K^{2}_{1}<\infty$.
Repeating this procedure, we get a projective resolution
$$\xymatrix@C=1cm{\cdots \ar[r]^(.3){\delta_{3}}& {0 \choose P_{2}}\oplus{Q_{2} \choose \mathbf{F}Q_{2}} \ar[r]^{\delta_{2}}
& {0 \choose P_{1}}\oplus{Q_{1} \choose \mathbf{F}Q_{1}} \ar[r]^{\delta_{1}}
& {0 \choose P_{0}}\oplus{Q_{0} \choose \mathbf{F}Q_{0}} \ar[r]^(.7){\delta_{0}} & {M_{1} \choose M_{2}} \ar[r] & 0 }$$
of ${M_{1} \choose M_{2}}$ in $(\mathbf{F},\mathcal{B})$
such that if ${K^{1}_{i} \choose K^{2}_{i}}$ is the kernel of $\delta_{i-1}$, then $\pd_{\mathcal{B}}K^{2}_{i}<\infty$.
Since $Q_{n+1}=0$, we have $\Ker \delta_{n} = {0 \choose K^{2}_{n+1}}$. As $\pd_{\mathcal{B}}K^{2}_{n+1}<\infty$,
we have a projective resolution
$$\xymatrix@C=1cm{
0 \ar[r] &  P_{n+m}\ar[r] & \cdots \ar[r]& P_{n+3} \ar[r] & P_{n+2} \ar[r]& P_{n+1} \ar[r]  & K^{2}_{n+1} \ar[r] & 0 }$$
of $K^{2}_{i+1}$ in $\mathcal{B}$, which induces the finite projective resolution
$$\xymatrix@C=1cm{
0 \ar[r] & {0 \choose P_{n+m}}\cdots \ar[r]& {0 \choose P_{n+3}}\ar[r] & {0 \choose P_{n+2}}\ar[r]
& {0 \choose P_{n+1}} \ar[r] & {0 \choose K^{2}_{n+1}} \ar[r] & 0 }$$
${0 \choose K^{2}_{n+1}}$ in $(\mathbf{F},\mathcal{B})$.
This means $\pd_{(\mathbf{F},\mathcal{B})}\Ker \delta_{n}=\pd_{(\mathbf{F},\mathcal{B})}{0 \choose K^{2}_{n+1}}<\infty$,
and hence $\pd_{(\mathbf{F},\mathcal{B})}{M_{1} \choose M_{2}}<\infty$.
\end{proof}

Motivated by the definition of compatible bimodules in \cite[Definition 1.1]{Z}, we introduce the following

\begin{defn}\label{def3.3}
The functor $\mathbf{F}$ is called $\emph{perfect}$ if the following two conditions are satisfied.
\begin{itemize}
\item[(P1)] If $\mathcal{Q^{\bullet}}$ is an exact sequence of projective objects in $\mathcal{A}$,
then $\mathbf{F}\mathcal{Q^{\bullet}}$ is exact.
\item[(P2)]If $\mathcal{P^{\bullet}}$ is a complete $\mathcal{B}$-projective resolution,
then $\Hom(\mathcal{P^{\bullet}},\mathbf{F}{Q})$ is exact for any $Q\in\mathcal{P(A)}$.
\end{itemize}
\end{defn}

For a ring $R$, $\Mod R$ is the category of left $R$-modules and $\mod R$ is the category of finitely generated left $R$-modules.
Let $\Lambda$ and $\Gamma$ be artin algebras. If $M$ is a compatible $(\Lambda,\Gamma)$-bimodule,
then the tensor functor $M\otimes_{\Gamma}-$ is perfect. Let $T:=\left(\begin{array}{cc}
\Lambda & M \\
0 & \Gamma\end{array}\right)$ be the upper triangular matrix algebra.
Then $\mod T$ is the left comma category $(M\otimes_{\Gamma}-,\mod \Lambda)$.

\begin{lem}\label{lem3.4}
The following statements are equivalent.
\begin{itemize}
\item[(1)] $\mathbf{F}$ satisfies {\rm (P2)};
\item[(2)] $\Ext_{\mathcal{B}}^{1}({G},\mathbf{F}{Q})=0$ for any ${G}\in \mathcal{GP}(\mathcal{B})$ and $Q\in\mathcal{P(A)}$;
\item[(3)] $\Ext_{\mathcal{B}}^{\geq 1}({G},\mathbf{F}{Q})=0$ for any ${G}\in \mathcal{GP}(\mathcal{B})$ and $Q\in\mathcal{P(A)}$.
\end{itemize}
\end{lem}

\begin{proof}
The implications $(1)\Rightarrow(3)\Rightarrow(2)$ are trivial.
Applying the functor $\Hom_{\mathcal{B}}(-,\mathbf{F}{Q})$ to a complete $\mathcal{B}$-projective resolution of $G$,
we get $(2)\Rightarrow(1)$.
\end{proof}

We now give an equivalent characterization of Gorenstein projective objects in the left comma category $(\mathbf{F},\mathcal{B})$.
It is a generalization of \cite[Theorem 1.4]{Z}.

\begin{thm}\label{th3.5}
If $\mathbf{F}$ is perfect,
then the following statements are equivalent for an object ${{Y} \choose{X}}_{\phi}$ in $(\mathbf{F},\mathcal{B})$.
\begin{itemize}
\item[(1)] ${{Y} \choose{X}}_{\phi}\in \mathcal{GP}((\mathbf{F},\mathcal{B}))$.
\item[(2)] $\phi:\mathbf{F}{Y}\rightarrow {X}$ is injective in $\mathcal{B}$,
$\Coker\phi\in\mathcal{GP}(\mathcal{B})$ and ${Y}\in\mathcal{GP}(\mathcal{A})$.
\end{itemize}
In this case, ${X}\in\mathcal{GP}(\mathcal{B})$ if and only if
$\mathbf{F}{Y}\in\mathcal{GP}(\mathcal{B})$.
\end{thm}

\begin{proof}
$(2)\Rightarrow (1)$ Assume that $\phi:\mathbf{F}{Y}\rightarrow {X}$ is injective in $\mathcal{B}$, $\Coker\phi\in\mathcal{GP}(\mathcal{B})$
and $Y\in\mathcal{GP}(\mathcal{A})$. Then we have a complete $\mathcal{A}$-projection resolution
\begin{center}$(\mathcal{Q^{\bullet}}, q^{\cdot}):=\xymatrix{\cdots \ar[r]&Q^{-1}\ar[r]&Q^{0}\ar[r]^{q^{0}}&Q^{1}\ar[r]&\cdots}$\end{center}
with $Y=\Ker q^{0}$. Since $\mathbf{F}\mathcal{Q^{\bullet}}$ is exact by (P1), we have the following exact sequence
\begin{center}$\xymatrix{0 \ar[r]&\mathbf{F}Y\ar[r]&\mathbf{F}Q^{0}\ar[r]^{Fq^{0}}&\mathbf{F}Q^{1}\ar[r]^{Fq^{1}}&\cdots}$.\end{center}
Since $\Coker\phi\in\mathcal{GP}(\mathcal{B})$, we have a complete $\mathcal{B}$-projective resolution
\begin{center}$\xymatrix{\cdots \ar[r]&P^{-1}\ar[r]&P^{0}\ar[r]^{d^{0}}&P^{1}\ar[r]&\cdots}$\end{center}
with $\Coker\phi=\Ker d^{0}$, so $\Ker d^{i}\in\mathcal{GP}(\mathcal{B})$, and hence $\Ext_{\mathcal{B}}^{1}(\Ker d^{i},\mathbf{F}Q^{i})=0$
for any $i\geq 0$. Applying Lemma \ref{lem2.1}(1) to the exact sequence
$$\xymatrix{0\ar[r]&\mathbf{F}Y\ar[r]^{\phi}&X\ar[r]&\Coker\phi\ar[r]&0},$$ we obtain an exact sequence
$$\xymatrix{0\ar[r]&X\ar[r]^{\partial^{-1}}&
P^{0}\oplus\mathbf{F}Q^{0}\ar[r]^{\partial^{0}}&P^{1}\oplus\mathbf{F}Q^{1}\ar[r]^{\partial^{1}}&\cdots}$$
with $\partial^{i}=\left(\begin{array}{cc}
 d^{i}&0\\
 \sigma^{i}&Fq^{i}
 \end{array}\right)$ and $\sigma^{i}:P^{i}\rightarrow FQ^{i+1}$ for any $i\geq 0$, such that the following diagram with exact rows
\begin{center}$\xymatrix{0\ar[r]&\mathbf{F}Y\ar[r]\ar[d]&\mathbf{F}Q^{0}\ar[r]\ar[d]&\mathbf{F}Q^{1}\ar[r]\ar[d]&\cdots\\
                        0\ar[r]&X\ar[r]&P^{0}\oplus\mathbf{F}Q^{0}\ar[r]&P^{1}\oplus\mathbf{F}Q^{1}\ar[r]&\cdots}$\end{center}
commutes. By a dual argument we get the following diagram with exact rows
\begin{center}$\xymatrix{\cdots\ar[r]&\mathbf{F}Q^{-2}\ar[r]\ar[d]&\mathbf{F}Q^{-1}\ar[r]\ar[d]&\mathbf{F}Y\ar[r]\ar[d]&0\\
                        \cdots\ar[r]&P^{-2}\oplus\mathbf{F}Q^{-2}\ar[r]&P^{-1}\oplus\mathbf{F}Q^{-1}\ar[r]&X\ar[r]&0.}$\end{center}
Combining these two diagrams to get the following diagram with exact rows
\begin{center}$\xymatrix{\cdots\ar[r]&\mathbf{F}Q^{-1}\ar[r]\ar[d]&\mathbf{F}Q^{0}\ar[r]\ar[d]&\mathbf{F}Q^{1}\ar[r]\ar[d]&\cdots\\
\cdots\ar[r]&P^{-1}\oplus\mathbf{F}Q^{-1}\ar[r]&P^{0}\oplus\mathbf{F}Q^{0}\ar[r]&P^{1}\oplus\mathbf{F}Q^{1}\ar[r]&\cdots.}$\end{center}
Actually, we have the following exact sequence of projective objects
\begin{center}$\xymatrix{L^{\bullet}=\cdots\ar[r]&{Q^{-1} \choose P^{-1}\oplus\mathbf{F}Q^{-1}}
\ar[r]&{Q^{0} \choose P^{0}\oplus\mathbf{F}Q^{0}}\ar[r]&{Q^{1} \choose P^{1}\oplus\mathbf{F}Q^{1}}\ar[r]&\cdots}$\end{center}
in $(\mathbf{F},\mathcal{B})$. Since each $L^{i}$ is a projective object in $(\mathbf{F},\mathcal{B})$,
applying $\Hom(L^{i},-)$ to the exact sequence:
\begin{center}$\xymatrix{0\ar[r]&{0 \choose \widetilde{P}\oplus \widetilde{FQ}}
\ar[r]&{\widetilde{Q} \choose \widetilde{P}\oplus \widetilde{FQ}}\ar[r]&{\widetilde{Q} \choose 0}\ar[r]&0}$\end{center}
we get the following exact sequence of complexes
$$ \xymatrix@C=0.5cm{
0 \ar[r] & \Hom(L^{\bullet},{0 \choose \widetilde{P}\oplus \widetilde{FQ}}) \ar[r]
& \Hom(L^{\bullet},{\widetilde{Q} \choose \widetilde{P}\oplus \widetilde{FQ}})
\ar[r] & \Hom(L^{\bullet},{\widetilde{Q} \choose 0}) \ar[r] & 0, }$$
that is,
$$ \xymatrix@C=0.5cm{
0 \ar[r] & \Hom(P^{\bullet},\widetilde{P}\oplus \widetilde{FQ}) \ar[r]
& \Hom(L^{\bullet},{\widetilde{Q} \choose \widetilde{P}\oplus \widetilde{FQ}}) \ar[r]
& \Hom(Q^{\bullet},\widetilde{Q}) \ar[r] & 0. }$$
Since $P^{\bullet}$ is a complete $\mathcal{B}$-projective resolution, it follows that $\Hom(P^{\bullet},\widetilde{P})$ is exact.
By (P2), $\Hom(P^{\bullet},\widetilde{FQ})$ is exact. Since $Q^{\bullet}$ is a complete $\mathcal{A}$-projective resolution,
it follows that $\Hom(Q^{\bullet},\widetilde{Q})$ is exact. Thus $\Hom(L^{\bullet},{\widetilde{Q} \choose \widetilde{P}\oplus \widetilde{FQ}})$
is also exact. Therefore we conclude that $L^{\bullet}$ is a complete $(\mathbf{F},\mathcal{B})$-projective resolution and
${{Y} \choose {X}}_{\phi}\in\mathcal{GP}((\mathbf{F},\mathcal{B}))$.

$(1)\Rightarrow (2)$ Let ${{Y} \choose {X}}_{\phi}\in\mathcal{GP}((\mathbf{F},\mathcal{B}))$.
Then we have a complete $(\mathbf{F},\mathcal{B})$-projective resolution
$$\xymatrix{L^{\bullet}:=\cdots\ar[r]&{Q^{-1} \choose P^{-1}\oplus\mathbf{F}Q^{-1}}\ar[r]
&{Q^{0} \choose P^{0}\oplus\mathbf{F}Q^{0}}\ar[r]^{d'^{0} \choose \partial^{0}}
&{Q^{1} \choose P^{1}\oplus\mathbf{F}Q^{1}}\ar[r]&\cdots}$$
such that $\Ker{d'^{0} \choose \partial^{0}}={{Y} \choose {X}}_{\phi}$. Then we get an exact sequence
$(Q^{\bullet},d'^{\bullet})$ of projective objects in $\mathcal{A}$ with $\Ker d'^{0}=Y$ and the following exact sequence
$$\xymatrix{V^{\bullet}:=\cdots\ar[r]&P^{-1}\oplus\mathbf{F}Q^{-1}\ar[r]
&P^{0}\oplus\mathbf{F}Q^{0}\ar[r]^{\partial^{0}}&P^{1}\oplus\mathbf{F}Q^{1}\ar[r]&\cdots}$$
with $\Ker\partial^{0}=X$. By (P1), $\mathbf{F}Q^{\bullet}$ is exact. Since
${d'^{i} \choose \partial^{i}}:{Q^{i} \choose P^{i}\oplus\mathbf{F}Q^{i}}\rightarrow{Q^{i+1} \choose P^{i+1}\oplus\mathbf{F}Q^{i+1}}$
is a morphism in $(\mathbf{F},\mathcal{B})$, we get that $\partial^{i}$ is of the form
$\partial^{i}=\left(\begin{array}{cc}
 d^{i}& 0 \\
\sigma^{i} & \mathbf{F}d'^{i}\end{array}\right)$
where $\sigma^{i}:P^{i}\rightarrow\mathbf{F}Q^{i+1}$ for any $i$.
We have the exact sequence of complexes
$$\xymatrix@C=0.5cm{
0 \ar[r] & \mathbf{F}Q^{\bullet} \ar[r] & V^{\bullet} \ar[r] & P^{\bullet} \ar[r] & 0 }$$ with $P^{\bullet}$ exact.
So we get the following diagram with exact columns and rows
$$\xymatrix{0\ar[r]&\mathbf{F}Q^{0}\ar[r]\ar[d]^{\mathbf{F}d'^{0}}&P^{0}\oplus\mathbf{F}Q^{0}\ar[r]\ar[d]^{\partial^{0}}&P^{0}\ar[r]\ar[d]^{d^{0}}&0\\
            0\ar[r]&\mathbf{F}Q^{1}\ar[r]\ar[d]^{\mathbf{F}d'^{1}}&P^{1}\oplus\mathbf{F}Q^{1}\ar[r]\ar[d]^{^{\partial^{1}}}&P^{1}\ar[r]\ar[d]^{d^{1}}&0\\
            0\ar[r]&\mathbf{F}Q^{2}\ar[r]\ar[d]&P^{2}\oplus\mathbf{F}Q^{2}\ar[r]\ar[d]&P^{2}\ar[r]\ar[d]&0\\
           &\vdots&\vdots&\vdots& }$$
such that $\Ker\mathbf{F}d'^{0}=\mathbf{F}Y$. Applying the snake lemma we get the following exact sequence
$$\xymatrix@C=0.5cm{
0 \ar[r] & \Ker \mathbf{F}d'^{0} \ar[r] &}\xymatrix@C=0.5cm{\Ker \partial^{0} \ar[r]
& \Ker d^{0} \ar[r] & \Im \mathbf{F}d'^{1} \ar[r] & \Im \partial^{1} \ar[r] &\Im d^{1}  \ar[r] & 0, }$$
that is, $$\xymatrix@C=0.5cm{
  0 \ar[r] & \mathbf{F}Y \ar[r]^{\phi} & X \ar[r]^{\pi'} & \Ker d^{0} \ar[r]
  & \Im \mathbf{F}d'^{1} \ar[r] &}\xymatrix@C=0.5cm{\Im \partial^{1} \ar[r] &\Im d^{1}  \ar[r] & 0. }$$
Because the morphism $\Im \mathbf{F}d'^{1}\rightarrow \Im \partial^{1}$ is injective, it follows that $\pi'$ is surjective.
Hence $\Ker d^{0}\cong \Coker \phi$.   Since $\Hom (L^{\bullet},{0 \choose \mathcal{P}})\cong \Hom(P^{\bullet},\mathcal{P})$
and $L^{\bullet}$ is a complete projection resolution, it follows that $\Hom(P^{\bullet},\mathcal{P})$ is exact.
Hence $P^{\bullet}$ is a complete $\mathcal{B}$-projective resolution and  $\Coker\phi\in\mathcal{GP}(\mathcal{B})$.
By (P2), $\Hom(P^{\bullet},\mathbf{F}Q)$ is exact. Similarly, since each $L^{i}$ is a projective object in
$(\mathbf{F},\mathcal{B})$, applying $\Hom(L^{i},-)$ to the exact sequence
\begin{center}$\xymatrix{0\ar[r]&{0 \choose \widetilde{P}\oplus \widetilde{FQ}}\ar[r]
&{\widetilde{Q} \choose \widetilde{P}\oplus \widetilde{FQ}}\ar[r]&{\widetilde{Q} \choose 0}\ar[r]&0,}$\end{center}
we get the following exact sequence of complexes
$$ \xymatrix@C=0.5cm{
  0 \ar[r] & \Hom(P^{\bullet},\widetilde{P}\oplus \widetilde{FQ}) \ar[r]
  & \Hom(L^{\bullet},{\widetilde{Q} \choose \widetilde{P}\oplus \widetilde{FQ}}) \ar[r]
  & \Hom(Q^{\bullet},\widetilde{Q}) \ar[r] & 0. }$$
Since $L^{\bullet}$ is a complete projective resolution,
$\Hom(L^{\bullet},{\widetilde{Q} \choose \widetilde{P}\oplus \widetilde{FQ}})$ is exact,
and then $\Hom(Q^{\bullet},\widetilde{Q})$ is also exact. It follows that $Y \in\mathcal{GP}(\mathcal{A})$.
\end{proof}

As an application of Theorem \ref{th3.5}, we have the following

\begin{cor}\label{cor3.6}
Let $\mathbf{F}$ be perfect. Then
\begin{itemize}
\item[(1)] If $(\mathbf{F},\mathcal{B})$ has finitely many isomorphism classes of indecomposable
Gorenstein projective objects, then so have $\mathcal{A}$ and $\mathcal{B}$.
\item[(2)] If $\mathcal{GP}(\mathcal{B})=\mathcal{P}(\mathcal{B})$, then
${0 \choose P}$ and ${Y \choose \mathbf{F}Y}$ are exactly all indecomposable Gorenstein projective
objects in $(\mathbf{F},\mathcal{B})$, where $Y$ runs over all indecomposable objects in
$\mathcal{GP}(\mathcal{A})$ and $P$ runs over all indecomposable objects in $\mathcal{P}(\mathcal{B})$.
\item[(3)] If $\mathcal{GP}(\mathcal{A})=\mathcal{P}(\mathcal{A})$,
then ${0 \choose X}$ and ${Q \choose \mathbf{F}Q}$ are exactly all indecomposable Gorenstein projective
objects in $(\mathbf{F},\mathcal{B})$, where $Q$ runs over all indecomposable objects in $\mathcal{P}(\mathcal{A})$
and $X$ runs over all the indecomposable objects in $\mathcal{GP}(\mathcal{B})$.
\end{itemize}
\end{cor}

\begin{proof}
$(1)$ Let $X\in\mathcal{GP}(\mathcal{B})$ and $Y\in\mathcal{GP}(\mathcal{A})$. Then by Theorem \ref{th3.5},
both ${0 \choose X}$ and ${Y \choose \mathbf{F}Y}$ are Gorenstein projective objects in $(\mathbf{F},\mathcal{B})$.
The assertion follows.

$(2)+(3)$ Let ${Y \choose X}_{\phi}$ be Gorenstein projective in $(\mathbf{F},\mathcal{B})$.
Then by Theorem \ref{th3.5}, there exists an exact sequence
$$\xymatrix{0\ar[r]&\mathbf{F}Y\ar[r]^{\phi}&X\ar[r]&\Coker\phi\ar[r]&0}$$
in $\mathcal{B}$ with $\Coker\phi\in\mathcal{GP}(\mathcal{B})$ and ${Y}\in\mathcal{GP}(\mathcal{A})$.

If $\mathcal{GP}(\mathcal{B})=\mathcal{P}(\mathcal{B})$, then $\Coker\phi \in\mathcal{P}(\mathcal{B})$
and the above exact sequence splits. If $\mathcal{GP}(\mathcal{A})=\mathcal{P}(\mathcal{A})$, then ${Y}\in\mathcal{P}(\mathcal{A})$.
By Lemma \ref{lem3.4}, we have $\Ext_{\mathcal{B}}^{\geq 1}({\Coker\phi},\mathbf{F}{Y})=0$.
So the above exact sequence also splits. So, in both cases, we have
$X=\mathbf{F}Y\oplus\Coker \phi$ and ${Y \choose X}_{\phi}={Y \choose \mathbf{F}Y}\oplus{0\choose \Coker \phi}$.
The assertions (2) and (3) follow.
\end{proof}

\begin{example}\label{exa3.7}
Let $k$ be a field and $T$ a finite-dimensional $k$-algebra given by the quiver
$$\xymatrix@C=15pt@R=8pt{&1\ar@(ur,ul)[]|{\gamma}\\
3\ar[r]&2\ar[u]&4\ar[l]}$$ with relation $\gamma^{3}=0$. Then
$$T=\left(\begin{array}{cc}e_{1}Te_{1} & e_{1}T(1-e_{1}) \\
0 & (1-e_{1})T(1-e_{1})\end{array}\right),$$
where $e_{1}$ is the idempotent corresponding to the vertex 1. We have that
$\Gamma:=(1-e_{1})T(1-e_{1})$ is a finite-dimensional $k$-algebra given by the quiver
$$\xymatrix@C=15pt@R=12pt{3\ar[r]&2&4\ar[l]},$$
and $\Lambda:= e_{1}Te_{1}$ is a finite-dimensional $k$-algebra given by the quiver
$$\xymatrix@C=15pt@R=12pt{1\ar@(ur,ul)[]|{\gamma}}$$
with relation $\gamma^{3}=0$.
Take $\mathcal{A}:=\mod \Gamma$, $\mathcal{B}:=\mod \Lambda$ and $\mathbf{F}:=M\otimes_{\Gamma}-$
with $M=e_{1}T(1-e_{1})$. Then $(\mathbf{F},\mathcal{B})=\mod T$.
We have $_{\Lambda}M \cong {_{\Lambda}\Lambda\oplus_{\Lambda}\Lambda\oplus_{\Lambda}\Lambda}$
and $M_{\Gamma} \cong I(2)_{\Gamma}\oplus I(2)_{\Gamma}\oplus I(2)_{\Gamma}$. Since $\Gamma$ is hereditary,
$\pd M_{\Gamma} \leq 1$ and $\mathbf{F}$ is perfect. Since $\Lambda$ is self-injective, each module in
$\mod \Lambda$ is Gorenstein projective. Then by Corollary \ref{cor3.6},
all indecomposable Gorenstein projective modules in $\mod T$ are as follows.
$$
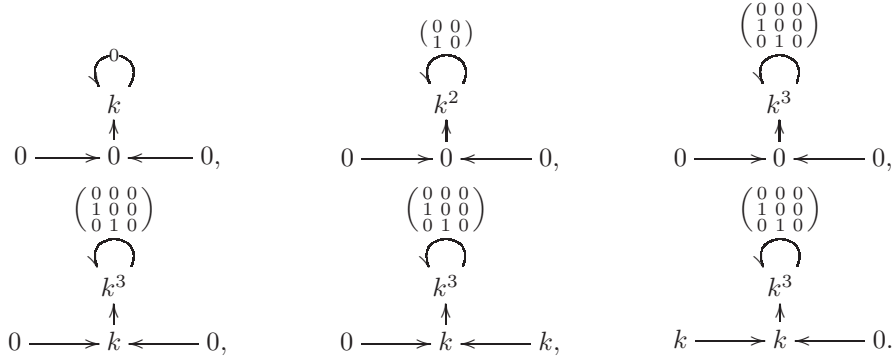
\begin{matrix}
\xymatrix@R=6pt{&k\ar@(ur,ul)[]|{0}\\
0\ar[r]&0\ar[u]&0,\ar[l]}  \ \ \ \ \ \ \ \
 &\xymatrix@R=6pt{&k^{2}\ar@(ur,ul)[]_{\left(\begin{smallmatrix}0&0\\1&0\end{smallmatrix}\right)}\\
0\ar[r]&0\ar[u]&0,\ar[l]}   \ \ \ \ \ \ \ \ &\xymatrix@R=6pt{&k^{3}\ar@(ur,ul)[]_{\left(\begin{smallmatrix}0&0&0\\1&0&0\\0&1&0\end{smallmatrix}\right)}\\
0\ar[r]&0\ar[u]&0,\ar[l]}\\
\xymatrix@R=6pt{&k^{3}\ar@(ur,ul)[]_{\left(\begin{smallmatrix}0&0&0\\1&0&0\\0&1&0\end{smallmatrix}\right)}\\
0\ar[r]&k\ar[u]&0,\ar[l]}  \ \ \ \ \ \ \ \ &\xymatrix@R=6pt{&k^{3}\ar@(ur,ul)[]_{\left(\begin{smallmatrix}0&0&0\\1&0&0\\0&1&0\end{smallmatrix}\right)}\\
0\ar[r]&k\ar[u]&k,\ar[l]}  \ \ \ \ \ \ \ \ &\xymatrix@R=6pt{&k^{3}\ar@(ur,ul)[]_{\left(\begin{smallmatrix}0&0&0\\1&0&0\\0&1&0\end{smallmatrix}\right)}\\
k\ar[r]&k\ar[u]&0.\ar[l]}
\end{matrix}$$
\end{example}

\begin{example}\label{exa3.8}
Let $k$ be a field and $T$ a finite-dimensional $k$-algebra given by the quiver
$$
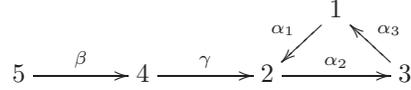
\xymatrix@C=15pt@R=12pt{&&&&&1\ar[ld]_{\alpha_{1}}\\
5\ar[rr]^{\beta}&&4\ar[rr]^{\gamma}&&2\ar[rr]^{\alpha_{2}}&&3\ar[lu]_{\alpha_{3}}}$$
with the relation $\alpha_{2}\alpha_{1}=\alpha_{3}\alpha_{2}=\alpha_{1}\alpha_{3}=0$. Then
$$T=\left(\begin{array}{cc}(e_{1}+e_{2}+e_{3})T(e_{1}+e_{2}+e_{3}) & (e_{1}+e_{2}+e_{3})T(e_{4}+e_{4}) \\
0 & (e_{4}+e_{5})T(e_{4}+e_{5})\end{array}\right),$$
where $e_{i}$ is the idempotent corresponding to the vertex $i$ for any $1\leq i\leq 5$. We have that
$\Gamma: = (e_{4}+e_{5})T(e_{4}+e_{5})$ is a finite-dimensional $k$-algebra given by the quiver
$$\xymatrix@C=15pt@R=12pt{5\ar[r]&4,}$$
and $\Lambda:= (e_{1}+e_{2}+e_{3})T(e_{1}+e_{2}+e_{3})$ is a finite-dimensional $k$-algebra
given by the quiver $$
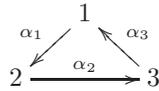
\xymatrix@C=15pt@R=12pt{&1\ar[ld]_{\alpha_{1}}\\
2\ar[rr]^{\alpha_{2}}&&3\ar[lu]_{\alpha_{3}}}$$
with relation $\alpha_{2}\alpha_{1}=\alpha_{3}\alpha_{2}=\alpha_{1}\alpha_{3}=0$.
Take $\mathcal{A}:=\mod \Gamma$, $\mathcal{B}:=\mod \Lambda$ and $\mathbf{F}:=M\otimes_{\Gamma}-$
with $M=(e_{1}+e_{2}+e_{3}) T(e_{4}+e_{5})$. Then $(\mathbf{F},\mathcal{B}) = \mod T$.
We have $_{\Lambda}M \cong {_{\Lambda}P(2)\oplus_{\Lambda}P(2)}$
and $M_{\Gamma} \cong P(5)_{\Gamma}\oplus P(5)_{\Gamma}$, and so $\mathbf{F}$ is perfect.
Notice that $\Lambda$ is self-injective and $\Gamma$ is hereditary, so by Corollary \ref{cor3.6},
all indecomposable Gorenstein projective modules in $\mod T$ are as follows.
$$
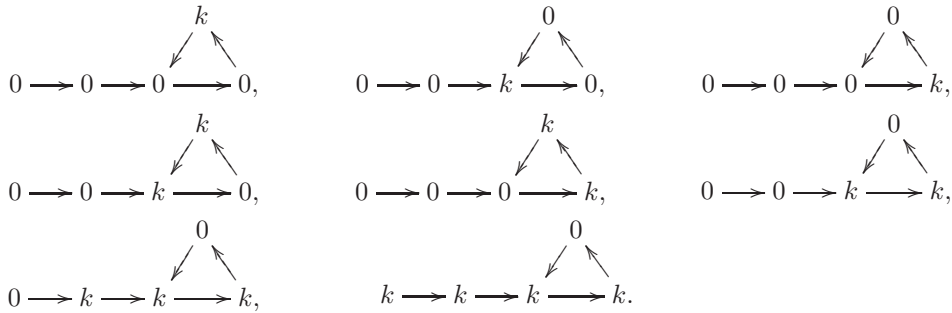
\begin{matrix}
\xymatrix@C=5pt@R=12pt{&&&&&k\ar[ld]\\
0\ar[rr]&&0\ar[rr]&&0\ar[rr]&&0,\ar[lu]} \ \ \ \ \ \ %
 &\xymatrix@C=5pt@R=12pt{&&&&&0\ar[ld]\\
0\ar[rr]&&0\ar[rr]&&k\ar[rr]&&0,\ar[lu]}  \ \ \ \ \ \ %
 &\xymatrix@C=5pt@R=12pt{&&&&&0\ar[ld]\\
0\ar[rr]&&0\ar[rr]&&0\ar[rr]&&k,\ar[lu]}\\ %
\xymatrix@C=5pt@R=12pt{&&&&&k\ar[ld]\\
0\ar[rr]&&0\ar[rr]&&k\ar[rr]&&0,\ar[lu]} \ \ \ \ \ \ %
 &\xymatrix@C=5pt@R=12pt{&&&&&k\ar[ld]\\
0\ar[rr]&&0\ar[rr]&&0\ar[rr]&&k,\ar[lu]} \ \ \ \ \ \ %
&\xymatrix@C=5pt@R=12pt{&&&&&0\ar[ld]\\
0\ar[rr]&&0\ar[rr]&&k\ar[rr]&&k,\ar[lu]}\\  %
\xymatrix@C=5pt@R=12pt{&&&&&0\ar[ld]\\
0\ar[rr]&&k\ar[rr]&&k\ar[rr]&&k,\ar[lu]} \ \ \ \ \ \ %
&\xymatrix@C=5pt@R=12pt{&&&&&0\ar[ld]\\
k\ar[rr]&&k\ar[rr]&&k\ar[rr]&&k.\ar[lu]}
\end{matrix}$$
\end{example}

By Theorem \ref{th3.5}, we also have the following

\begin{cor}\label{cor3.9}
If $\mathbf{F}$ is perfect,
then $\mathcal{GP}((\mathbf{F},\mathcal{B}))=\mathcal{P}((\mathbf{F},\mathcal{B}))$
if and only if $\mathcal{GP}(\mathcal{A})=\mathcal{P}(\mathcal{A})$
and $\mathcal{GP}(\mathcal{B})=\mathcal{P}(\mathcal{B})$.
\end{cor}

\begin{proof}
We first prove the necessity.
Let $Y$ be Gorenstein projective in $\mathcal{A}$. Then by Theorem \ref{th3.5},
${Y \choose \mathbf{F}Y}$ is Gorenstein projective in $(\mathbf{F},\mathcal{B})$. So
${Y \choose \mathbf{F}Y}$ is projective in $(\mathbf{F},\mathcal{B})$ by assumption, and hence
$Y$ is projective in $\mathcal{A}$. Now let $X$ be Gorenstein projective in $\mathcal{B}$.
Then by Theorem \ref{th3.5}, ${0 \choose X}$ is Gorenstein projective in $(\mathbf{F},\mathcal{B})$.
So ${0 \choose X}$ is projective in $(\mathbf{F},\mathcal{B})$ by assumption, and hence $X$ is projective
in $\mathcal{B}$.

We next prove the sufficiency. Let ${Y \choose X}_{\phi}$ be Gorenstein projective in
$(\mathbf{F},\mathcal{B})$. Then we have the following exact sequence
$$\xymatrix@C=15pt{0\ar[r]&\mathbf{F}Y}\xymatrix@C=15pt{\ar[r]&X\ar[r]&\Coker \phi \ar[r]&0}$$
in $\mathcal{B}$ with $\Coker \phi \in\mathcal{GP}(\mathcal{B})$ and $Y\in\mathcal{GP(A)}$
by Theorem \ref{th3.5}. So $\Coker \phi $ is projective in $\mathcal{B}$ and $Y$ is projective
in $\mathcal{A}$ by assumption, and hence
$X=\mathbf{F}Y\oplus\Coker \phi$ and ${Y \choose X}_{\phi}={Y \choose \mathbf{F}Y}\oplus{0\choose \Coker \phi}$.
Thus ${Y \choose X}_{\phi}$ is projective in $(\mathbf{F},\mathcal{B})$ by Lemma \ref{lem3.1}.
\end{proof}

Recall from \cite{ECT} that a ring $R$ is called {\it strongly left CM-free} if each Gorenstein projective
module in $\Mod R$ is projective.
Let $\Lambda$ and $\Gamma$ be arbitrary rings and $M$ a $(\Lambda,\Gamma)$-bimodule,
and let $T:=\left(\begin{array}{cc}
\Lambda & M \\
0 & \Gamma\end{array}\right)$ be the upper triangular matrix ring.
Then $\Mod T$ is the left comma category $(M\otimes_{\Gamma}-,\Mod \Lambda)$.
If $M_{\Gamma}$ has finite flat dimension and $_{\Lambda}M$ has finite projective dimension,
then the functor $M\otimes_{\Gamma}-$ is perfect. So, as an immediate consequence of Corollary \ref{cor3.9},
we have the following

\begin{cor}\label{cor3.10}
Let $\Lambda$ and $\Gamma$ be arbitrary rings and $M$ a $(\Lambda,\Gamma)$-bimodule,
and let $T$ be the upper triangular matrix ring as above.
If $M_{\Gamma}$ has finite flat dimension and $_{\Lambda}M$ has finite projective dimension,
then $T$ is strongly left CM-free if and only if so are $\Lambda$ and $\Gamma$.
\end{cor}

The above corollary generalized \cite[Theorem 4.1]{ECT}, where the assumption that
$\Lambda$ is left Gorenstein regular is needed.


\section{Recollements}

\begin{defn}(\cite{FP,Ps})\label{def4.1}
 A recollement, denoted by ($\mathcal{A},\mathcal{B},\mathcal{C}$), of abelian categories is a diagram
$$\xymatrix{\mathcal{A}\ar[rr]!R|{i_{*}}&&\ar@<-2ex>[ll]!R|{i^{*}}\ar@<2ex>[ll]!R|{i^{!}}\mathcal{B}
\ar[rr]!L|{j^{*}}&&\ar@<-2ex>[ll]!L|{j_{!}}\ar@<2ex>[ll]!L|{j_{*}}\mathcal{C}}$$
of abelian categories and additive functors such that
\begin{enumerate}
\item[(1)] ($i^{*},i_{*}$), ($i_{*},i^{!}$), ($j_{!},j^{*}$) and ($j^{*},j_{*}$) are adjoint pairs.
\item[(2)] $i_{*}$, $j_{!}$ and $j_{*}$ are fully faithful.
\item[(3)] $\Im i_{*}=\Ker j^{*}$.
\end{enumerate}
\end{defn}

The following lemma is fundamental in this section.

\begin{lem}{\rm (\cite[Example 2.12]{Ps})}\label{lem4.2}
There exists the following recollement of abelian categories:
$$\xymatrix{\mathcal{B}\ar[rr]!R|(.4){i_{*}}&&\ar@<-2ex>[ll]!R|(.6){i^{*}}\ar@<2ex>[ll]!R|(.6){i^{!}}(\mathbf{F},\mathcal{B})
\ar[rr]!L|(.6){j^{*}}&&\ar@<-2ex>[ll]!L|(.4){j_{!}}\ar@<2ex>[ll]!L|(.4){j_{*}}\mathcal{A},}$$
where $$i^{*}:{{Y} \choose{X}}_{\phi} \mapsto \Coker\phi,\ i_{*}:{X} \mapsto {{0} \choose{X}},\
i^{!}:{{Y} \choose{X}} \mapsto {X},$$ $$j_{!}:{Y} \mapsto {{Y} \choose{\mathbf{F}{Y}}},\
j^{*}:{{Y} \choose{X}} \mapsto {Y},\ j_{*}:{Y} \mapsto {{Y} \choose{0}}.$$
\end{lem}

\begin{defn}(\cite{BBD})\label{def4.3}
Let $\mathcal{C'}$, $\mathcal{C}$ and $\mathcal{C''}$ be triangulated categories. The diagram of exact functors
\begin{align}\label{re1}
\xymatrix{\mathcal{C'}\ar[rr]!R|{i_{*}}&&\ar@<-2ex>[ll]!R|{i^{*}}\ar@<2ex>[ll]!R|{i^{!}}\mathcal{C}
\ar[rr]!L|{j^{*}}&&\ar@<-2ex>[ll]!L|{j_{!}}\ar@<2ex>[ll]!L|{j_{*}}\mathcal{C''}}\end{align}
is a {\it recollement} of $\mathcal{C}$ relative to $\mathcal{C'}$ and $\mathcal{C''}$, if the following four conditions are satisfied.
\begin{enumerate}
\item[(R1)] ($i^{*},i_{*}$), ($i_{*},i^{!}$), ($j_{!},j^{*}$) and ($j^{*},j_{*}$) are adjoint pairs.
\item[(R2)] $i_{*}$, $j_{!}$ and $j_{*}$ are fully faithful.
\item[(R3)] $ j^{*}i_{*}= 0 $.
\item[(R4)] For each object $X \in \mathcal{C}$, the counits and units give rise to the following distinguished triangles
 $$\xymatrix@C=15pt{j_{!}j^{*}(X)\ar[r]^(.7){\epsilon_{X}}&
X\ar[r]^(.4){\eta_{X}}&i_{*}i^{*}(X)\ar[r]&j_{!}j^{*}(X)[1],}$$
$$\xymatrix@C=15pt{i_{*}i^{!}(X)\ar[r]^(.7){\omega_{X}}&X\ar[r]^(.4){\zeta_{X}}&
j_{*}j^{*}(X)\ar[r]&i_{*}i^{!}(X)[1],}$$
where $[1]$ is the shift functor.
\end{enumerate}
A {\it left recollement} of $\mathcal{C}$ relative to $\mathcal{C'}$ and $\mathcal{C''}$
is a diagram of exact functors consisting of the upper two rows in the diagram (4.1) satisfying all the conditions which
involve only the functors $i^*,i_*,j_{!},j^{*}$.
\end{defn}

The following result is useful in the sequel.

\begin{lem}{\rm (\cite[Section 1]{IKM})}\label{lem4.5}
 Let (\ref{re1}) be a diagram of triangulated categories. Then the following statements are equivalent.
\begin{itemize}
\item[(1)] The diagram (\ref{re1}) is a recollement.
\item[(2)] The conditions (R1), (R2) and $\Im i_{*} = \Ker j^{*}$ are satisfied.
\item[(3)] The conditions (R1), (R2) and $\Im j_{!} = \Ker i^{*}$ are satisfied.
\item[(4)] The conditions (R1), (R2) and $\Im j_{*} = \Ker i^{!}$ are satisfied.
\end{itemize}
\end{lem}

\begin{rem}\label{rem4.5}
Each assertion in Lemma \ref{lem4.5}(2)--(4) involving only the functors $i^*,i_*,j_{!},j^{*}$ is equivalent to that
the upper two rows in the diagram (\ref{re1}) is a left recollement.
\end{rem}

The following result is a generalization of \cite[Theorem 3.3]{Z}.

\begin{thm}\label{th4.4}
If $\mathbf{F}$ is perfect, then there exists a left recollement
$$\xymatrix{\underline{\mathcal{GP}(\mathcal{B})}\ar@<-1ex>[rr]!R|(.4){i_{*}}
&&\ar@<-1ex>[ll]!R|(.6){i^{*}}\ar@<-1ex>[rr]!R|(.4){j^{*}}\underline{\mathcal{GP}((\mathbf{F},\mathcal{B}))}
&&\ar@<-1ex>[ll]!R|(.6){j_{!}}\underline{\mathcal{GP}(\mathcal{A})}.}$$
\end{thm}

\begin{proof}
We first construct the functors involved. By Theorem \ref{th3.5}, we know the form of
Gorenstein projective objects in $(\mathbf{F},\mathcal{B})$.
If a morphism ${X \choose Y}_{\phi}\rightarrow{X' \choose Y'}_{\phi'}$ factors through a
projective object ${0 \choose P}\oplus{Q \choose \mathbf{F}Q}$, then we have the following diagram with exact rows
$$\xymatrix{0\ar[r]&\mathbf{F}X\ar[r]^{\phi}\ar[d]&Y\ar[r]\ar[d]&\Coker \phi\ar[r]\ar[d]&0\\
            0\ar[r]&\mathbf{F}Q\ar[r]\ar[d]&\mathbf{F}Q\oplus P\ar[r]\ar[d]&P\ar[r]\ar[d]&0\\
            0\ar[r]&\mathbf{F}X'\ar[r]^{\phi'}&Y'\ar[r]&\Coker \phi'\ar[r]&0.}$$
Hence the functor $i^{*}$ in Lemma \ref{lem4.2} induces a functor which we still denote by
$i^{*}:\underline{\mathcal{GP}((\mathbf{F},\mathcal{B}))}\rightarrow \underline{\mathcal{GP}(\mathcal{B})}$.

By Lemma \ref{lem4.2}, we have the functor $i_{*}$ given by $Y\rightarrow {0 \choose Y}$.
It is obvious a functor $\mathcal{GP}(\mathcal{B})\rightarrow \mathcal{GP}((\mathbf{F},\mathcal{B}))$.
If a morphism $Y\rightarrow Y'$ in $\mathcal{B}$ factors through a projective object $P$,
then ${0 \choose Y}\rightarrow {0 \choose Y'}$ factors through a projective object ${0 \choose P}$
in $(\mathbf{F},\mathcal{B})$. Hence $i_{*}$ induces a functor $i_{*}:\underline{\mathcal{GP}(\mathcal{B})}
\rightarrow \underline{\mathcal{GP}((\mathbf{F},\mathcal{B}))}$, which is fully faithful.

By Lemma \ref{lem4.2}, we have the functor $j_{!}$ given by
$A\rightarrow {A \choose \mathbf{F}A}$. It is a functor
$\mathcal{GP}(\mathcal{A})\rightarrow \mathcal{GP}((\mathbf{F},\mathcal{B}))$ by Theorem \ref{th3.5}.
If a morphism $X\rightarrow X'$ in $\mathcal{A}$ factors through a projective object $Q$,
then ${X \choose \mathbf{F}X}\rightarrow {X' \choose \mathbf{F}X'}$ factors through a projective object
${Q \choose \mathbf{F}Q}$ in $(\mathbf{F},\mathcal{B})$.
Hence $j_{!}$ induces a functor $j_{!}:\underline{\mathcal{GP}(\mathcal{A})}\rightarrow
\underline{\mathcal{GP}((\mathbf{F},\mathcal{B}))}$, which is fully faithful.

By Lemma \ref{lem4.2}, we have the functor $j^{*}$ given by
${X \choose Y}\rightarrow X $. It is a functor from $\mathcal{GP}((\mathbf{F},\mathcal{B}))\rightarrow
\mathcal{GP}(\mathcal{A})$ by Theorem \ref{th3.5}. If a morphism ${X \choose Y}\rightarrow {X' \choose Y'}$
in $(\mathbf{F},\mathcal{B})$ factors through a projective object ${Q \choose \mathbf{F}Q}\oplus {0 \choose P}$,
then $X \rightarrow X'$ factors through a projective object $Q$ in $\mathcal{A}$.
Hence $j^{*}$ induces a functor $j^{*}:\underline{\mathcal{GP}((\mathbf{F},\mathcal{B}))}\rightarrow \underline{\mathcal{GP}(\mathcal{A})}$.
It follows easily from \cite[Chapter I, Section 2]{H} that $i_{*}, j^{*}$ constructed above are exact functors.
By Lemma \ref{lem4.2}, we have that both ($i^{*},i_{*}$) and($j_{!},j^{*}$) are adjoint pairs. Thus
$i^{*}$ and $j_{!}$ are exact functors by \cite[Lemma 8.3]{K}.


By construction, we have $\Im i_{*} \subseteq \Ker j^{*}$ and $\Ker j^{*} = \{{X \choose Y}\in
\underline{\mathcal{GP}((\mathbf{F},\mathcal{B}))}\mid X \in \mathcal{P(A)}\}$. Let ${X \choose Y} \in \Ker j^{*} $.
By Theorem \ref{th3.5}, we have the following exact sequence
$$\xymatrix@C=15pt{0\ar[r]&\mathbf{F}X\ar[r]^{\phi}&Y\ar[r]&\Coker \phi \ar[r]&0}$$
in $\mathcal{B}$ with $\Coker\phi\in\mathcal{GP}(\mathcal{B})$. Then $\Ext_{\mathcal{B}}^{1} (\Coker\phi,\mathbf{F}X)=0$
by Lemma \ref{lem3.4}. So the above exact splits and $Y\cong\mathbf{F}X \oplus \Coker\phi$. Thus we have
$${X \choose Y} \cong
{X \choose \mathbf{F}X}\oplus {0\choose \Coker\phi}= i_{*}(\Coker\phi),$$
which implies $\Ker j^{*}\subseteq\Im i_{*}$.

Finally, applying Lemma \ref{lem4.5}(2) and Remark \ref{rem4.5}, we get the required left recollement.
\end{proof}

It is natural to ask when the left recollement in Theorem \ref{th4.4} can be filled into a recollement. In the following,
we will study this question.

Recall from \cite{BR} that
an abelian category $\mathcal{B}$ with enough projective and injective objects is called {\it Gorenstein} if
$\pd_{\mathcal{B}}\mathcal{I(B)}<\infty$ and $\id_{\mathcal{B}}\mathcal{P(B)}<\infty$.

\begin{lem}\label{lem4.6}
Let $\mathbf{F}$ be perfect. If $\mathcal{B}$ is Gorenstein and $\mathbf{F}$ preserves projectives, then
$\mathbf{F}$ preserves Gorenstein projectives.
\end{lem}

\begin{proof}
Let $Y\in\mathcal{A}$ be Gorenstein projective. Then there exists a complete $\mathcal{P(A)}$-resolution
$$\xymatrix@C=1cm{
\mathrm{Q^{\bullet}}: = \cdots \to  Q_{1} \to Q_{0} \buildrel {d} \over \longrightarrow Q^{0} \to Q^{1} \to \cdots }$$
in $\mathcal{A}$ such that $Y\cong\Im d$.
Since $\mathbf{F}$ is perfect, $\mathbf{F}Q^{\bullet}$ is exact and $\mathbf{F}Y\cong \Ker\mathbf{F}d$.
If $\mathbf{F}$ preserves projectives, then all terms in $\mathbf{F}Q^{\bullet}$ are projective in $\mathcal{B}$.
Let $P\in\mathcal{B}$ be projective. Because $\mathcal{B}$ is Gorenstein by assumption, we have $\id_{\mathcal{B}}P<\infty$.
So $\Hom(\mathbf{F}Q^{\bullet},P)$ is exact, and hence $\mathbf{F}Y$ is Gorenstein projective.
\end{proof}

As a generalization of \cite[Theorem 3.5]{Z}, we have the following

\begin{thm}\label{thm4.7}
Let $\mathbf{F}$ be perfect. If $\mathcal{B}$ is Gorenstein and $\mathbf{F}$ preserves projectives,
then there exists a recollement
$$\xymatrix{\underline{\mathcal{GP}(\mathcal{B})}\ar[rr]!R|(.4){i_{*}}&&\ar@<-2ex>[ll]!R|(.6){i^{*}}
\ar@<2ex>[ll]!R|(.6){i^{!}}\underline{\mathcal{GP}((\mathbf{F},\mathcal{B}))}
\ar[rr]!L|(.6){j^{*}}&&\ar@<-2ex>[ll]!L|(.4){j_{!}}\ar@<2ex>[ll]!L|(.4){j_{*}}\underline{\mathcal{GP}(\mathcal{A})}.}$$
\end{thm}

\begin{proof}
By Theorem \ref{th4.4}, there exists the following left recollement
$$\xymatrix{\underline{\mathcal{GP}(\mathcal{B})}\ar@<-1ex>[rr]!R|(.4){i_{*}}
&&\ar@<-1ex>[ll]!R|(.6){i^{*}}\ar@<-1ex>[rr]!R|(.4){j^{*}}\underline{\mathcal{GP}((\mathbf{F},\mathcal{B}))}
&&\ar@<-1ex>[ll]!R|(.6){j_{!}}\underline{\mathcal{GP}(\mathcal{A})}.}$$
By Lemma \ref{lem4.2}, we have the functor $i^{!}$ given by ${X \choose Y}\rightarrow Y$. It is
a functor $\mathcal{GP}((\mathbf{F},\mathcal{B}))\to\mathcal{GP}(\mathcal{B})$. If a morphism
${X \choose Y}_{\phi}\rightarrow{X' \choose Y'}_{\phi'}$ in $(\mathbf{F},\mathcal{B})$ factors through
a projective object ${0 \choose P}\oplus{Q \choose\mathbf{F}Q}$ with $Q$ projective in $\mathcal{A}$
and $P$ projective in $\mathcal{B}$, then $Y \rightarrow Y'$ factors through $P\oplus\mathbf{F}Q$.
Since $\mathbf{F}$ preserves projectives, we have that $\mathbf{F}Q$ is projective in $\mathcal{B}$, and so
$P\oplus\mathbf{F}Q$ is also projective in $\mathcal{B}$. Hence $i^{!}$ induces
a functor $i^{!}:\underline{\mathcal{GP}((\mathbf{F},\mathcal{B}))}\rightarrow\underline{\mathcal{GP}(\mathcal{B})}$.
By Lemma \ref{lem4.2}, ($i_{*},i^{!}$) is an adjoint pair.

We claim that there
exists a fully faithful functor
$j_{*}:\underline{\mathcal{GP}(\mathcal{A})}\rightarrow
\underline{\mathcal{GP}((\mathbf{F},\mathcal{B}))}$ given by $X\rightarrow{X \choose
P}$ with $P\in \mathcal{P}({\mathcal{B}})$, such that there exists an exact sequence
$$\xymatrix@C=15pt{0\ar[r]&\mathbf{F}X\ar[r]^{\phi}&P\ar[r]&\Coker \phi \ar[r]&0}$$
in $\mathcal{B}$ with $\Coker \phi \in \mathcal{GP}(\mathcal{B})$.

Let $X \in \mathcal{GP}(\mathcal{A})$. By Lemma \ref{lem4.6}, $\mathbf{F}X \in
\mathcal{GP}(\mathcal{B})$ and there exists an exact sequence
$$\xymatrix@C=15pt{0\ar[r]&\mathbf{F}X\ar[r]^{\phi}&P\ar[r]&\Coker \phi \ar[r]&0}$$
in $\mathcal{B}$ with $P \in \mathcal{P}({\mathcal{B}})$ and $\Coker \phi \in
\mathcal{GP}(\mathcal{B})$. Let $g:X \rightarrow X'$ be a morphism in
$\mathcal{GP}(\mathcal{B})$ and $P' \in \mathcal{P}({\mathcal{B}})$ such that
$$\xymatrix@C=15pt{0\ar[r]&\mathbf{F}X'\ar[r]^{\phi'}&P'\ar[r]&\Coker \phi'\ar[r]&0}$$
is an exact sequence in $\mathcal{B}$ with $\Coker \phi' \in \mathcal{GP}(\mathcal{B})$. Since
$\Ext_{\mathcal{B}}^{1} (\Coker\phi, P') = 0 $, we have the following diagram with exact rows
$$\xymatrix@C=15pt{0\ar[r]&\mathbf{F}X\ar[r]^{\phi}\ar[d]^{\mathbf{F}g}&P\ar[r]
\ar@{-->}[d]^{f}&\Coker \phi \ar[r]\ar@{-->}[d]&0\\
0\ar[r]&\mathbf{F}X'\ar[r]^{\phi'}&P'\ar[r]&\Coker \phi' \ar[r]&0.}$$
If there exists a morphism $f':P\rightarrow P'$ such that $f'\phi = \phi'\mathbf{F}g$,
then $f'-f$ factors through $\Coker \phi$. Since $\Coker \phi \in\mathcal{GP}(\mathcal{B}) $,
we have a monomorphism $\rho:\Coker \phi\rightarrow \widetilde{P}$ with $\widetilde{P}$ projective in $\mathcal{B}$.
Then we easily see that ${g \choose f}-{g \choose f'}$ factors through the projective object ${0 \choose \widetilde{P}}$
in $(\mathbf{F},\mathcal{B})$ and hence ${g \choose f} = {g \choose f'}$ in $\underline{\mathcal{GP}((\mathbf{F},\mathcal{B}))}$.
Note that if we take $g=\id_{X}$, this also proves that the object ${X \choose P}\in \underline{\mathcal{GP}((\mathbf{F},\mathcal{B}))}$
is independent of the choice of $P$. Thus we get a functor
$j_{*}':\mathcal{GP}(\mathcal{B})\rightarrow\underline{\mathcal{GP}((\mathbf{F},\mathcal{B}))}$.

Assume that $g:X \rightarrow X'$ in $\mathcal{A}$ factors through a projective object $Q$ with
$g =g_{2}g_{1} $. Since $\mathbf{F}Q$ is projective in $\mathcal{B}$ by assumption, it is injective in
$\mathcal{GP}(\mathcal{B})$, therefore there exists a morphism
$\alpha:P\rightarrow\mathbf{F}Q$ such that $\mathbf{F}g_{1} = \alpha \phi $. Since
$(f-\phi'\mathbf{F}g_{2}\alpha)\phi=0$, there exists $\widetilde{f}:\Coker
\phi\rightarrow P'$ such that $(f-\phi'\mathbf{F}g_{2}\alpha)= \widetilde{f}\pi$.
Let $\eta:\Coker \phi\rightarrow P_{1}$ be a monomorphism with $P_{1}\in
\mathcal{P}(\mathcal{B}) $. Then we get $\beta:P_{1}\rightarrow P'$ such that
$\widetilde{f}=\beta\eta$. Thus ${g \choose f}$ factors through the projective
object ${Q \choose \mathbf{F}Q}\oplus{0 \choose P_{1}}$ in $(\mathbf{F},\mathcal{B})$ with ${g \choose
f}={g_{2} \choose (\phi'\mathbf{F}g_{2},\beta)}{g_{1}\choose {\alpha \choose
\eta\pi}}$. Therefore $j_{*}'$ induces a functor
$j_{*}:\underline{\mathcal{GP}(\mathcal{B})}\rightarrow\underline{\mathcal{GP}((\mathbf{F},\mathcal{B}))}$
which given by $X\rightarrow{X \choose P}$ and $g\rightarrow{g \choose f}$.
If ${g \choose f}$ factors through a projective object ${0 \choose P}\oplus{Q \choose \mathbf{F}Q}$
in $(\mathbf{F},\mathcal{B})$, then $g$ factors through the projective object $Q$. Thus $j_{*}$ is fully faithful.
The claim is proved.

Let ${g \choose f}:{X \choose Y}\rightarrow{X' \choose P}$ be a morphism in
$\mathcal{GP}((\mathbf{F},\mathcal{B}))$. By Theorem \ref{th3.5}, there exists an exact sequence
$$\xymatrix@C=15pt{0\ar[r]&\mathbf{F}X'\ar[r]^{\phi}&P\ar[r]&\Coker \phi \ar[r]&0}$$
in $\mathcal{B}$ with $P$ projective and $\Coker\phi\in\mathcal{GP}(\mathcal{B})$. Then ${g \choose
f}$ factors through the projective object ${0 \choose P'}\oplus{Q \choose \mathbf{F}Q}$
in $(\mathbf{F},\mathcal{B})$ if and only if $g:X\rightarrow X'$ factors through the projective object $Q$ in $\mathcal{B}$.
It follows that the isomorphism
$$\Hom_{\underline{\mathcal{GP}(\mathcal{A})}}(X,X')\cong
\Hom_{\underline{\mathcal{GP}((\mathbf{F},\mathcal{B}))}}({X \choose Y},{X' \choose P})$$
is natural in both variables and ($j^{*},j_{*}$) is an adjoint pair.

Finally, applying Lemma \ref{lem4.5}, we get the required recollement.
\end{proof}

\vspace{0.5cm}

\textbf{Acknowledgement.} This research was partially supported by NSFC (Grant No. 11571164)
and a Project Funded by the Priority Academic Program Development of Jiangsu Higher Education Institutions.

\end{document}